\documentclass[12pt]{amsart}
\usepackage{amsmath,latexsym,amscd,amsbsy,amssymb,amsfonts,amsthm,fleqn,leqno,
euscript, graphicx, texdraw, pb-diagram}
\usepackage[matrix,arrow,curve]{xy}

\numberwithin{equation}{section}
\newtheorem{thm}{Theorem}[section]
\newtheorem{pro}[thm]{Proposition}
\newtheorem{lem}[thm]{Lemma}
\newtheorem{cor}[thm]{Corollary}

\newtheorem*{MainTheorem1}{Theorem \ref{main1}}

\newtheorem*{MainTheorem2}{Theorem \ref{main2}}

\theoremstyle{definition}

\theoremstyle{remark}

\newcommand{\Tor}{\mathrm{Tor}}
\newcommand{\Ext}{\mathrm{Ext}}
\newcommand{\Hom}{\mathrm{Hom}}

\newcommand{\bd}{\mathrm{bd}}

\newcommand{\Int}{\mathrm{Int}}

\numberwithin{equation}{section}

\begin{document}
%%%%%%%%%%% Begin Topmatter %%%%%%%%%%%%%%%%%

\title[Homological dimension and homogeneous ANR spaces]
{Homological dimension and homogeneous ANR spaces}

\author{V. Valov}
\address{Department of Computer Science and Mathematics,
Nipissing University, 100 College Drive, P.O. Box 5002, North Bay,
ON, P1B 8L7, Canada} \email{veskov@nipissingu.ca}

\thanks{The author was partially supported by NSERC
Grant 261914-13.}


\begin{abstract}
The homological dimension $d_G$ of metric compacta was introduced by Alexandroff in \cite{a}. In this paper we provide some general properties of $d_G$, mainly with an eye towards describing the dimensional full-valuedness of compact metric spaces. As a corollary of the established properties of $d_G$, we prove that any two-dimensional $lc^2$ metric compactum is dimensionally full-valued. This improves the well known result of Kodama \cite{ko1} that every two-dimensional $ANR$ is dimensionally full-valued.
Applications for homogeneous metric $ANR$-compacta are also given.
\end{abstract}
\maketitle\markboth{}{Homological dimension}
%%%%%%%%%% End topmatter %%%%%%%%%%%%%%%%%%%%%

%%%%%%%%%%%%%%%%%%%%%%%%%%%%%%%%%%%%%%%%%%%%%%%%%%%%%%%%%%%%%%%%
%%%%%%%%%%%%%%%%%%%%%%%%%%%%%%%%%%%%%%%%%%%%%%%%%%%%%%%%%%%%%%%%

%%%%%%%%%%%%%%%%%% TABLE OF CONTENT %%%%%%%%%%%%%%%%%%%%%%%%%%%%%%%%%%

%\tableofcontents

%%%%%%%%%%%%%%%%%%%%%%%%%%%%%%%%%%%%%%%%%%%%%%%%%%%%%%%%%%%%%%%%%%%
%%%%%%%%%%%%%%%%%%%%%%%%%%%%%%%%%%%%%%%%%%%%%%%%%%%%%%%%%%%%%%%%%%%%%%

\section{Introduction}
Reduced \v{C}ech homology $H_n(X;G)$ and cohomology groups $H^n(X;G)$ with coefficient from an abelian group $G$ are considered everywhere below.
By a space, unless stated otherwise, we mean a metric compactum.

Alexandroff \cite{a} introduced the homological dimension theory as a further geometrization of the ordinary dimension theory. His definition of the homological dimension $d_GX$ of a space $X$ (with coefficients in an abelian group $G$) provided a new characterization of the covering dimension $\dim X$: If $X$ is finite-dimensional, then $\dim X$ is the maximum integer $n$ such that $\dim X=d_{\mathbb Q_1}X=d_{\mathbb S^1}X=n$, where $\mathbb S^1$ is the circle group and $\mathbb Q_1$ is the group of rational elements of $\mathbb S^1$. Here is Alexandroff's definition: the homological dimension $d_GX$ is the maximum integer $n$ such that  there exist a closed set $F\subset X$ and a nontrivial element $\gamma\in H_{n-1}(F;G)$ with $\gamma$ being $G$-homologous to zero in $X$. According to \cite[p.207]{a} we have $d_GX\leq\dim X$ for any  space $X$.

The aim of this paper is to provide more properties of the homological dimension. Except some general properties of $d_G$, we also describe dimensional full-valuedness of compact metric spaces using the equalities $d_GX=\dim X$ for some specific groups $G$. The lack of exactness of \v{C}ech homology is a real problem. We overcome this problem by using the exact homology, developed by Sklyarenko \cite{sk}, and its coincidence with \v{C}ech homology in some important cases (let us note that the exact homology is equivalent to Steenrod homology in the class of metric compacta).

Because the definition of $d_GX$ does not provide any information for the homology groups $H_{k-1}(F;G)$ when $F\subset X$ is closed and $k<d_GX-1$, we consider
the set $\mathcal{H}_{X,G}$ of all integers $k\geq 1$ such that there exists a closed set $F\subset X$ and a nontrivial element $\gamma\in H_{k-1}(F;G)$, which is homologous to zero in $X$. Obviously, $d_GX=\max\mathcal{H}_{X,G}$. The set $\mathcal{H}_{X,G}$ was very useful in establishing some local homological properties of homogeneous $ANR$-compacta, see \cite{vv0}. In this paper we also provide some properties of $\mathcal{H}_{X,G}$.

Suppose $(K,A)$ is a pair of closed subsets of a space $X$ with $A\subset K$. By $i^n_{A,K}$ we denote the homomorphism from $H_n(A;G)$ into $H_n(K;G)$ generated by the inclusion $A\hookrightarrow K$. This inclusion generates also the homomorphism $j_{K,A}^n:H^{n}(K;G)\to H^{n}(A;G)$.
Following \cite{bb}, we say that
$K$ is an {\em $n$-homology membrane spanned on $A$ for some $\gamma\in H_n(A;G)$} provided $\gamma$ is homologous to
zero in $K$, but not homologous to zero in any proper closed subset of $K$ containing $A$. It is well known that for space $X$ and a closed set $A\subset X$ the existence a non-trivial element $\gamma\in H_n(A;G)$ with $i^n_{A,X}(\gamma)=0$ yields the existence of a closed set $K\subset X$ containing $A$ such that $K$ is an $n$-homology membrane for $\gamma$ spanned on $A$, see \cite[Property 5., p.103]{bb}. The above statement remains true for any Hausdorff compactum $X$ and a closed set $A\subset X$, see Lemma \ref{membrane} below.

Recall that the cohomological dimension $\dim_GX$ is the largest integer $n$ such that there exists a closed set $A\subset X$ with $H^n(X,A;G)\neq 0$. It is well known that $\dim_GX\leq n$ iff every map $f\colon A\to K(G,n)$ can be extended to a map $\tilde f\colon X\to K(G,n)$, where $K(G,n)$ is the Eilenberg-MacLane space of type $(G,n)$, see \cite{sp}. We also say that a space $X$ is dimensionally full-valued if $\dim X\times Y=\dim X+\dim Y$ for all metric compacta $Y$, or equivalently (see \cite{bo} and \cite[Theorem 11]{ku}), $\dim_GX=\dim_{\mathbb Z}X$ for any abelian group $G$.

The paper is organised as follows.
In Section 2 we provide some properties of the dimension $d_G$ and the set $\mathcal{H}_{X,G}$. As stated above, we generally have the inequalities
$d_{\mathbb Z}X\leq\dim X\geq d_{G}X$, where the equalities yield peculiar properties of $X$ stated below:
\begin{MainTheorem2}
The following holds for any space $X$ with $\dim X=n$:
\begin{itemize}
\item[(1)] If $d_{G}X=n$ for a torsion free abelian group $G$, then:
\item[(1.1)] there exists a point $x\in X$ and a local base $\mathcal B_x$ at $x$ consisting of open sets $U$ such that $\bd\,\overline U$ is a dimensionally full-valued set of dimension $n-1$, and
\item[(1.2)] $X$ is dimensionally full-valued;
%\item[(2)] If $d_{G}X=n$, where $G$ is any abelian group with a trivial torsion part, then $X$ is dimensionally full-valued;
\item[(2)] If $X$ is dimensionally full-valued and $X$ has a base of $\mathcal B$ of open sets $U$ with $H^{n}(\overline U;\mathbb Z)=0$, then
$d_{G}X=n$ for every field $G$;
\item[(3)] If $X$ is dimensionally full-valued and $lc^{n}$, then $d_{\mathbb Z}X=d_{G}X=n$ for every field $G$.
\end{itemize}
\end{MainTheorem2}

The assumption in item (1)of the above theorem that $G$ is torsion free is essential. Indeed, we already observed that $\dim X=\dim_{\mathbb Q_1}$ for any finite-dimensional metric compactum $X$, but not any such $X$ is dimensionally full-valued.

It follows from Theorem 2.9 (see Corollary 2.10) that if $X$ is a space such that $\dim X=n$ and $X$ is homologically locally connected up to dimension $n$ with respect to the singular homology with integer coefficients (br., $lc^n$), then the following conditions are equivalent:  $X$ is dimensionally full-valued; there is a point $x\in X$ having a local base $\mathcal B_x$ of open sets $U$ such that $\bd\,U$ is dimensionally full-valued for each $U\in\mathcal B_x$;
$d_GX=n$ for any torsion free field $G$; $d_{\mathbb Z}X=n$.

Let us also mention Corollary 2.11 stating that every two-dimensional $lc^{2}$-space is dimensionally full-valued. This improves a result of  Kodama \cite[Theorem 8]{ko1} that every two-dimensional $ANR$ is dimensionally full-valued.

Using the sets $\mathcal{H}_{X,G}$, we obtain in Section 3 some properties of homogeneous $ANR$-spaces. In particular, the following proposition is shown:
\begin{MainTheorem1}
Let $X$ be a homogeneous $ANR$-space.
\begin{itemize}
\item[(1)] If $G$ is a field and $\dim_GX=n$, then $n\in\mathcal H_{X,G}$ and $n+1\not\in\mathcal H_{X,G}$;
\item[(2)] If $\dim X=n$, then $\mathcal H_{X,\mathbb Z}\cap\{n-1,n\}\neq\varnothing$.
\end{itemize}
\end{MainTheorem1}

The final Section 4 contains all preliminary information about the exact and \v{C}ech homology groups we are using in this paper.

\section{Some properties of $d_G$}

Everywhere below $\widehat{H}_*$ denotes the exact homology (see \cite{ma}, \cite{sk}), where $G$ is any module over a commutative ring with unity.
This homology is equivalent to Steenrod's homology \cite{st} in the category of
compact metric spaces. According to \cite{sk}, for every module $G$
and a compact metric pair $(X,A)$ there exists a natural transformation $T_{X,A}:\widehat{H}_{*}(X,A;G)\to H_{*}(X,A;G)$ between the exact and \v{C}ech homologies such that $T^k_{X,A}:\widehat{H}_{k}(X,A;G)\to H_{k}(X,A;G)$, is a surjective homomorphism for each $k$, see \cite[Theorem 4]{sk} (if $A$ is the empty set, we denote $T^k_{X,\varnothing}$  by
$T^k_{X}$). This homomorphism is an isomorphism in some situations, see Proposition 4.4.

We already mentioned that $d_GX\leq\dim X$ for any group $G$. Next proposition shows that in some cases we have a stronger inequality.
\begin{pro}
For any $X$ and a field $G$ we have $d_GX\leq\dim_GX$.
\end{pro}

\begin{proof}
It suffices to prove the required inequality when $\dim_GX=n<\infty$.
The set $\mathcal H_{X,G}$ does not contain any $k\geq n+2$ because, according to the Universal Coefficient Theorem for \v{C}ech homology with field coefficients, $k-1>\dim_GX$ implies $H_{k-1}(F;G)=0$ for every closed set $F\subset X$. So, we need to show that $n+1\not\in\mathcal H_{X,G}$. Indeed, if $n+1\in\mathcal H_{X,G}$, then there exits a closed set $A\subset X$ and
 non-trivial $\gamma\in H_{n}(A;G)$ with $i^n_{A,X}(\gamma)=0$. Since $G$ is a field, by Proposition 4.4(i),
$H_{n}(A;G)$ and $H_{n}(X;G)$ are isomorphic to the groups
$\widehat{H}_{n}(A;G)$ and $\widehat{H}_{n}(X;G)$, respectively. Let $\widehat{\gamma}$ be the element from $\widehat{H}_{n}(A;G)$  corresponding to $\gamma$ and $\widehat{i}^n_{A,X}:\widehat{H}_{n}(A;G)\to\widehat{H}_{n}(X;G)$ be the inclusion homomorphism. It follows from the exact sequence
$$\to\widehat{H}_{n+1}(X,A;G)\to\widehat{H}_{n}(A;G)\to\widehat{H}_{n}(X;G)\to $$ that $\widehat{H}_{n+1}(X,A;G)\neq 0$. On the other hand, $\dim_GX=n$ implies  $H^{n+2}(X,A;G)=0$, so $\mathrm{Ext_G}(H^{n+2}(X,A;G),G)=0$. Hence, by \cite[Theorem 3]{sk}, we have the exact sequence
$$0\to\widehat{H}_{n+1}(X,A;G)\to\mathrm{Hom_G}(H^{n+1}(X,A;G),G)\to 0.$$
Then the groups
 $\widehat{H}_{n+1}(X,A;G)$ and $\mathrm{Hom_G}(H^{n+1}(X,A;G),G)$ are isomorphic. Finally, since $\dim_GX=n$, $H^{n+1}(X,A;G)=0$. Thus, $\widehat{H}_{n+1}(X,A;G)$ is also trivial, a contradiction.
\end{proof}

We consider the set $\mathcal{CH}_{X,G}$ of all integers $k\geq 1$ such that there exist a closed set $F\subset X$ and a nontrivial element $\alpha\in H^{k-1}(F;G)$ with $\alpha\not\in j^{k-1}_{X,F}(H^{k-1}(X;G))$ (in such a situation we say that $\alpha$ is not extendable over $X$). Obviously, the set
$\mathcal{CH}_{X,G}$ is a dual notion of the set $\mathcal{H}_{X,G}$.

Next proposition provides a connection between the sets $\mathcal{CH}_{X,G}$ and $\mathcal{H}_{X,G^*}$, where $G^*$ denotes the dual group of $G$ ($G$ is considered as a topological group with the discrete topology).
\begin{pro}
Let $X$ be a metric compactum and $G$ a countable abelian group. Then $n\in\mathcal{CH}_{X,G}$ if and only if $n\in\mathcal{H}_{X,G^*}$.
\end{pro}

\begin{proof}
The proof follows from the fact that for any space $Z$ and any $k\geq 0$ the group $H_{k}(Z;G^*)$ is the dual group of $H^{k}(Z;G)$, where both $G$ and $H^{k}(Z;G)$ are equipped with discrete topology, see \cite[Chap.VIII,\S4]{hw}. Indeed,
if $n\in\mathcal{CH}_{X,G}$, then there exist a closed set $A\subset X$ and a non-trivial element $\gamma\in H^{n-1}(A;G)$ with
$\gamma\not\in j^{n-1}_{X,A}(H^{n-1}(X;G))$. Choose a homomorphism $\alpha: H^{n-1}(A;G)\to\mathbb S^1$ such that $\alpha(\gamma)\neq 0$ and
$\alpha\big(j^{n-1}_{X,A}(H^{n-1}(X;G))\big)=0$. Because of the duality mentioned above, $\alpha$ can be treated as a non-trivial element of $H_{n-1}(A;G^*)$ such that $i^{n-1}_{A,X}=\alpha\circ j^{n-1}_{X,A}$. So, $i^{n-1}_{A,X}(\alpha)=0$, which yields  $n\in\mathcal{H}_{X,G^*}$.

If $n\in\mathcal{H}_{X,G^*}$, we take a closed set $A\subset X$ and a non-trivial $\alpha\in H_{n-1}(A;G^*)$ with $i^{n-1}_{A,X}(\alpha)=0$. We consider $\alpha$ as a non-trivial element of $H^{n-1}(A;G)^*$. So, $\alpha(\gamma)\neq 0$ for some $\gamma\in H^{n-1}(A;G)$. This implies that
$\gamma\not\in j^{n-1}_{X,A}(H^{n-1}(X;G))$ because $\alpha$ is homologous to zero in $X$.
\end{proof}

\begin{pro}
For every space $X$ and every abelian group $G$ we have $\max\mathcal{CH}_{X,G}=\dim_GX$.
\end{pro}

\begin{proof}
The definition of the set $\mathcal{CH}_{X,G}$ and the exact sequence
$$
\xymatrix{
\to H^{n-1}(X;G)\to H^{n-1}(F;G)\to H^{n}(X,F;G)\to\\
}
$$
imply that if $n\in\mathcal{CH}_{X,G}$, then $H^{n}(X,F;G)\neq 0$. Thus, $\max\mathcal{CH}_{X,G}\leq\dim_GX$.

On the other hand, if $\dim_GX=n$, then there exists a closed set $Y\subset X$ and points $x\in Y$ possessing a basis of open in $Y$ sets $W$ such that all homomorphisms $j^{n-1}_{\overline W,bd_YW}: H^{n-1}(\overline W;G)\to H^{n-1}(bd_YW;G)$ are not surjective, see \cite[Theorem 2]{ku}. Hence, the homomorphisms
$j^{n-1}_{X,bd_YW}:H^{n-1}(X;G)\to H^{n-1}(bd_YW;G)$ are also not surjective, which implies that $n\in\mathcal{CH}_{X,G}$.
\end{proof}

\begin{cor}
For every countable abelian group $G$ and a space $X$ we have:
\begin{itemize}
\item[(1)] $d_{G^*}X=\dim_GX$;
\item[(2)] If $\dim_GX$ is finite, then $\mathcal{H}_{X,G^*}=[1,\dim_GX]$.
\end{itemize}
\end{cor}

\begin{proof}
The first item follows from Propositions 2.2 - 2.3. Obviously, the second item also follows from Propositions 2.2 - 2.3 provided
$\mathcal{CH}_{X,G}=[1,\dim_GX]$. The last equality follows from the inclusion $[1,\dim_GX]\subset\mathcal{CH}_{X,G}$. So, let $n\in [1,\dim_GX]$. Since $\dim_GX > n-1$, there exists a closed set$A\subset X$ and
a non-trivial $\gamma\in H^{n-1}(A;G)$ with $\gamma\not\in j^{n-1}_{X,A}(H^{n-1}(X;G))$ (we use the following well known fact: $\dim_GX\leq k$ if and only if
the homomorphism $j^k_{X,F}:H^k(X;G)\to H^k(F;G)$ is surjective for every closed set $F\subset X$). Hence, $[1,\dim_GX]\subset\mathcal{CH}_{X,G}$.
\end{proof}

The well known Bockstein theorem (see \cite{dr}, \cite{ku}) states that for every group $G$ there is a set of countable groups $\sigma(G)$ such that $\dim_GX=\sup\{\dim_HX:H\in\sigma(G)\}$ for any space $X$. Corollary 2.4(1) yields a similar equality for $d_{G^*}X$.

\begin{cor}
For any space $X$ and any countable abelian group $G$ we have $d_{G^*}X=\sup\{d_{H^*}X:H\in\sigma(G)\}$
\end{cor}

Next proposition is a non-metrizable homological analogue of Theorem 14 from \cite[chapter IV,\S6]{a} and Property 5 from \cite{bb}.
\begin{lem}\label{membrane}
Let $B\subset A$ be a compact pair and $H_{n-1}(B;G)$ contains a non-trivial element $\gamma$, which is homologous to zero in $A$. Then there exists a closed set $K\subset A$ containing $B$ such that $K$ is a homological membrane for $\gamma$.
\end{lem}

\begin{proof}
 Consider the family $\mathcal F$ of all closed sets $F\subset A$ containing $B$ such that $i_{B,F}^{n-1}(\gamma)=0$. Obviously, $A\in\mathcal F$ and any minimal element of $\mathcal F$ is a homological membrane for $\gamma$. To show that $\mathcal F$ has a minimal element, according to Zorn's lemma, we need the following claim.

 \smallskip
\textit{Claim $1$.
Suppose $\{F_\alpha:\alpha\in\Lambda\}$ is an infinite linearly ordered decreasing subfamily of $\mathcal F$ and $F=\bigcap F_\alpha$. Then $F\in\mathcal F$.}

\smallskip
Denote by $\mathcal C_1$ the family of all finite open covers of $F_1$ and let $\mathcal C_\alpha$, $\mathcal C$ be the restrictions of $\mathcal C_1$ on the sets $F_\alpha$ and $F$, respectively. The set $\mathcal C_1$ becomes directed with respect to the relation: $\omega'\prec\omega''$ iff $\omega''$ is a refinement of $\omega'$. For every $\omega\in\mathcal C_1$  let $\mathrm{N}_\omega$ be the nerve of $\omega$. There is a natural simplicial map
$p_{\omega'',\omega'}:\mathrm{N}_{\omega''}\to\mathrm{N}_{\omega'}$ provided $\omega'\prec\omega''$. The maps $p_{\omega'',\omega'}$ induce corresponding
homomorphisms $\pi_{\omega'',\omega'}^k:H_k(\mathrm{N}_\omega'';G)\to H_k(\mathrm{N}_\omega';G)$ for every $k\geq 0$. Therefore,
we have the inverse systems $\mathcal S_1^k=\{H_k(\mathrm{N}_\omega;G),\pi_{\omega',\omega}^k:\omega\prec\omega'\}$, $k\geq 0$, such that $H_k(F_1;G)$ is the limit of $\mathcal S_1^k$, see \cite{es}. The sets $\mathcal C$ and $\mathcal C_\alpha$ also generates the nerves $\mathrm{N}_{\omega|F}$, $\mathrm{N}_{\omega|F_\alpha}$ and the inverse systems $$\mathcal S^k=\{H_k(\mathrm{N}_{\omega|F};G),\pi_{\omega'|F,\omega|F}^k:\omega\prec\omega'\}$$ and
$$\mathcal S_\alpha^k=\{H_k(\mathrm{N}_{\omega|F_\alpha};G),\pi_{\omega'|F_\alpha,\omega|F_\alpha}^k:\omega\prec\omega'\},$$
where $\omega|F$ and $\omega|F_\alpha$ denote the restrictions of $\omega$ on the sets $F$ and $F_\alpha$, respectively. It is easily seen that for every open cover $\delta$ of $F_\alpha$ there exists $\omega\in\mathcal C_1$ such that $\omega|F_\alpha$ refines $\delta$. Consequently, $H_k(F_\alpha;G)=\varprojlim\mathcal S^k_\alpha$. Similarly, $H_k(F;G)=\varprojlim\mathcal S^k$. Note that each $\mathrm{N}_{\omega|F_\alpha}$ is a subcomplex of the finite complex $\mathrm{N}_\omega$, $\omega\in\mathcal C_1$. Moreover, if $F_\alpha\subset F_\beta$ for some $\alpha,\beta\in\Lambda$, then $\mathrm{N}_{\omega|F_\alpha}$ is a subcomplex of $\mathrm{N}_{\omega|F_\beta}$. So, the families $\{\mathrm{N}_{\omega|F_\alpha}:\alpha\in\Lambda\}$, $\omega\in\mathcal C_1$, are decreasing and each one of them consists of finite simplicial complexes.
Because $\Lambda$ is infinite, for every $\omega\in\mathcal C_1$ there is $\alpha(\omega)\in\Lambda$ such that $\mathrm{N}_{\omega|F_\beta}=\mathrm{N}_{\omega|F_{\alpha(\omega)}}=\mathrm{N}_{\omega|F}$ for all $\beta\succ\alpha(\omega)$.

Now, we can complete the proof of Claim 1. Let $i_{B,F}^{n-1}(\gamma)=\gamma_F$ and $\omega\in\mathcal C_1$. Then, according to the above notations,
$\mathrm{N}_{\omega|F}=\mathrm{N}_{\omega|F_{\alpha(\omega)}}$. Consider also the projections $\pi_{F,\omega}:H_{n-1}(F;G)\to H_{n-1}(\mathrm{N}_{\omega|F};G)$ and $\pi_{F_{\alpha(\omega)},\omega}:H_{n-1}(F_{\alpha(\omega)};G)\to H_{n-1}(\mathrm{N}_{\omega|F_{\alpha(\omega)}};G)$. Then we have the commutative diagram { $$
\begin{CD}
H_{n-1}(F;G)@>{{i^{n-1}_{F,F_{\alpha(\omega)}}}}>>H_{n-1}(F_{\alpha(\omega)};G)\\
@VV{\pi_{F,\omega}}V@VV{\pi_{F_{\alpha(\omega)},\omega}}V\\
H_{n-1}(\mathrm{N}_{\omega|F};G)@>{{i^{n-1}_{\mathrm{N}_{\omega|F},\mathrm{N}_{\omega|F_{\alpha(\omega)}}}}}>> H_{n-1}(\mathrm{N}_{\omega|F_{\alpha(\omega)}};G)@.
\end{CD}
$$}\\

Since $i^{n-1}_{F,F_{\alpha(\omega)}}(\gamma_F)=i_{B,F_\alpha(\omega)}^{n-1}(\gamma)=0$ and $i^{n-1}_{\mathrm{N}_{\omega|F},\mathrm{N}_{\omega|F_{\alpha(\omega)}}}$ is the identity, we obtain $\pi_{F,\omega}(\gamma_F)=0$. The last equality (being true for each $\omega\in\mathcal C_1$) implies $\gamma_F=0$. Hence, $F\in\mathcal F$.
\end{proof}

\begin{cor}
Let $X$ be a Hausdorff compactum with $n\in\mathcal{H}_{X,G}$. Then there exists a pair $F\subset K$ of closed subsets of $X$ such that $K$ is a homological membrane for some non-trivial element of $H_{n-1}(F;G)$.
\end{cor}

\begin{proof}
Since $n\in\mathcal{H}_{X,G}$, there exists a proper closed subset $F\subset X$ and a non-trivial $\gamma\in H_{n-1}(F;G)$ such that $\gamma$ is $G$-homologous to zero in $X$. Then apply Lemma \ref{membrane} for the pair $F\subset X$.
\end{proof}

\begin{cor}
Let $X$ be a space and $G$ any group. Then,
$n\in\mathcal{H}_{X,G}$ if and only if there exists $x\in X$ and a local base $\mathcal B_x$ at $x$ consisting of open sets $U$ with each
$\bd\,\overline U$ containing a closed set $F_U$ such that $i^{n-1}_{F_U,\overline U}(\gamma_U)=0$ for some non-trivial $\gamma_U\in H_{n-1}(F_U;G)$.
\end{cor}

\begin{proof} Obviously, the existence of a local base with the stated property yields $n\in\mathcal{H}_{X,G}$. So, we need to show the "only if" part. To this end,
let $n\in\mathcal{H}_{X,G}$. Then, there exists a compact pair $F\subset K$ such that $K$ is an $(n-1)$-homology membrane for some $\gamma\in H_{n-1}(F;G)$ (see Lemma 2.6).  According to \cite[Property 6]{bb}, every $x\in K\setminus F$ has a local base $\mathcal B_x$ of open sets $U\subset X$ such that $\overline{U\cap K}$ is an $(n-1)$-homology membrane for an element $\gamma_U\in H^{n-1}(F_U;G)$ spanned on $F_U$, where $F_U=\bd_K(U\cap K)$. Because $i^{n-1}_{F_U,\overline{U\cap K}}(\gamma_U)=0$, $i^{n-1}_{F_U,\overline U}(\gamma_U)=0$.
\end{proof}

A space $X$ is {\em homologically locally connected in dimension $n$} (br., $n-lc$) if for every $x\in X$ and a neighborhood $U$ of $x$ in $X$ there exists a neighborhood $V\subset U$ of $x$ such that the homomorphism $\widetilde{i}_{V,U}^n:\widetilde{H}_n(V;\mathbb Z)\to\widetilde{H}_n(U;\mathbb Z)$ is trivial, where $\widetilde{H}_*(.;\mathbb Z)$ denotes the singular homology groups with integer coefficients ($\widetilde{H}_*(.;\mathbb Z)$ should not be confused with the reduced singular homology). If in the above definition the group $Z$ is replaced by a group $G$, we say that $X$ is $n-lc$ with respect to $G$. One can show that $X\in n-lc$ with respect to any group $G$
provided $X\in k-lc$ for every $k\in\{n,n-1\}$ (for $n\geq 1$ this follows from Proposition 4.8; for $n=0$ it is obvious).
 We say that $X$ is {\em homologically locally connected up to dimension $n$} (br., $X\in lc^n$) provided $X$ is $k-lc$ for all $k\leq n$.

According to \cite[Theorem 1]{mar} and \cite[Proposition 9]{sk}, the homology groups $H_k(X,A;G)$, $\widehat{H}_k(X,A;G)$ and $\widetilde{H}_k(X,A;G)$ are isomorphic for any abelian group $G$ and any $k\leq n$ provided $A\subset X$ is a pair of metric compacta such that both $X$ and $A$ are $lc^{n+1}$, see Propositions 4.2 and 4.3.

It follows from Proposition 4.8 and Proposition 4.2 that any $lc^n$ space $X$ is homologically locally connected up to dimension $n$ with respect to both singular and \v{C}ech homology with arbitrary coefficients, i.e. for every group $G$, a point $x\in X$ and its neighborhood $U$ there exists a neighborhood $V$ of $x$ such that $V\subset U$ and the homomorphisms
$i^k_{V,U}:H_k(V;G)\to H_k(U;G)$ and $\widetilde{i}^k_{V,U}:\widetilde{H}_k(V;G)\to\widetilde{H}_k(U;G)$  are trivial for all $k\leq n$.

\begin{thm}\label{main2}
The following holds for any space $X$ with $\dim X=n$:
\begin{itemize}
\item[(1)] If $d_{G}X=n$ for a torsion free abelian group $G$, then:
\item[(1.1)] there exists a point $x\in X$ and a local base $\mathcal B_x$ at $x$ consisting of open sets $U$ such that $\bd\,\overline U$ is a dimensionally full-valued set of dimension $n-1$, and
\item[(1.2)] $X$ is dimensionally full-valued;
%\item[(2)] If $d_{G}X=n$, where $G$ is any abelian group with a trivial torsion part, then $X$ is dimensionally full-valued;
\item[(2)] If $X$ is dimensionally full-valued and $X$ has a base of $\mathcal B$ of open sets $U$ with $H^{n}(\overline U;\mathbb Z)=0$, then
$d_{G}X=n$ for every field $G$;
\item[(3)] If $X$ is dimensionally full-valued and $lc^{n}$, then $d_{\mathbb Z}X=d_{G}X=n$ for every field $G$.
\end{itemize}
\end{thm}

\begin{proof}
$(1.1)$ Suppose $d_{G}X=n$. Then $n\in\mathcal{H}_{X,G}$, so there exists $x\in X$ and a local base $\mathcal B_x$ at $x$ of open sets $U\subset X$ such that each
$\bd\,\overline U$ contains a closed set $F_U$ with $i^{n-1}_{F_U,X}(\gamma_U)=0$ for some non-trivial $\gamma_U\in H_{n-1}(F_U;G)$ (see Corollary 2.8). According to Proposition 4.5, $H_{n-1}(F_U;\mathbb Z)\neq 0$.
Because $\dim X=n$, we can suppose that $\dim\bd\,\overline U\leq n-1$ for all $U\in\mathcal B_x$. Hence, $\dim F_U\leq n-1$. On the other hand, $H_{n-1}(F_U;G)\neq 0$ implies $\dim F_U\geq n-1$. Thus, $\dim F_U= n-1$ and, by Proposition 4.7, $F_U$ is dimensionally full-valued. This yields that $bd\,\overline U$ is also dimensionally
full valued. Indeed, for any finite-dimensional metric compactum $Y$, we have the inequalities
$$\dim Y+n-1=\dim F_U\times Y\leq\dim bd\,\overline U\times Y\leq\dim Y+n-1.$$
Hence, $\dim bd\,\overline U\times Y=\dim bd\,\overline U+\dim Y$.

$(1.2)$ As above, we can find a closed set $F_U\subset X$ of dimension $n-1$ and a non-trivial $\gamma\in H_{n-1}(F_U;G)$ such that
$i_{F_U,X}^{n-1}(\gamma)=0$. Then, by Proposition 4.6, $H_n(X,F_U;G)\neq0$. Because $G$ is torsion free, the last relation implies
$H_n(X,F_U;\mathbb Z)\neq0$, see Proposition 4.5. Finally, according to Proposition 4.7, $X$ is dimensionally full-valued.

$(2)$ Let $G$ be any field. Because $X$ is dimensionally full-valued, $\dim_GX=\dim X=n$, see \cite[Theorem 11]{ku}.
Since $d_GX\leq\dim X=n$, it suffices to show that $n\in\mathcal H_{X,G}$. To this end, observe that, since $G$ is a field,
$h\dim_GX=\dim_GX=n$, see \cite{ha}. Here,
$h\dim_\mathbb QX$ denotes the maximal integer $m$ such that $\widehat{H}_{k}(X,P;G)=0$ for all $k>m$ and all closed sets $P\subset X$.
On the other hand, by \cite[Corollary 2]{sk1}, if $h\dim_GX$ is finite, it is equal to the maximum integer $k$ such that the module
$H^x_k=\varinjlim_{x\in U} \widehat{H}_k(X,X\setminus U;G)$ is non-trivial for some $x\in X$. Therefore, there exist a point $x$ with $\widehat{H}_n(X,X\setminus U;G)\neq 0$ for all sufficiently small neighborhoods $U$ of $x$. We fix such $U$ with $U\in\mathcal B$.
Then, by the excision axiom, $\widehat{H}_n(X,X\setminus U;G)$ is isomorphic to $\widehat{H}_n(\overline U,bd\,\overline U;G)$.
Consider the exact sequences (see Theorems 2-3 from \cite{sk})
{ $$
\begin{CD}
\widehat{H}_{n}(\overline U;G)@>{{}}>>\widehat{H}_{n}(\overline U,\bd\,\overline U;G)@>{{\partial}}>>\widehat{H}_{n-1}(\bd\,\overline U;G)@>{{\widehat{i}^{n-1}_{\bd\,\overline U,\overline U}}}>>\widehat{H}_{n-1}(\overline U;G)
\end{CD}
$$}\\
and
$$0\to\mathrm{Ext}(H^{n+1}(\overline U;\mathbb Z),G)\to\widehat{H}_{n}(\overline U;G)\to\mathrm{Hom}(H^{n}(\overline U;\mathbb Z),G)\to 0.$$
Since $H^{n+1}(\overline U;\mathbb Z)=0$ (recall that $\dim \overline U\leq n$) and $H^{n}(\overline U;\mathbb Z)=0$, the second sequence yields $\widehat{H}_{n}(\overline U;G)=0$.
Hence, the homomorphism $\partial:\widehat{H}_{n}(\overline U,bd\,\overline U;G)\to\widehat{H}_{n-1}(bd\,\overline U;G)$ is injective and
the image $L=\partial(\widehat{H}_{n}(\overline U,bd\,\overline U;G))$ is a non-trivial subgroup of $\widehat{H}_{n-1}(bd\,\overline U;G)$ such that $\widehat{i}^{n-1}_{bd\,\overline U,\overline U}(L)=0$, where $\widehat{i}^{n-1}_{bd\,\overline U,\overline U}:\widehat{H}_{n-1}(bd\,\overline U;G)\to \widehat{H}_{n-1}(\overline U;G)$ is the homomorphism generated by the inclusion $bd\,\overline U\hookrightarrow\overline U$.
Because $G$ is a field, the homomorphisms $T^{n-1}_{\overline U}:\widehat{H}_{n-1}(\overline U;G)\to H_{n-1}(\overline U;G)$ and
$T^{n-1}_{\bd\,\overline U}:\widehat{H}_{n-1}(\bd\,\overline U;G)\to H_{n-1}(\bd\,\overline U;G)$ are isomorphisms, see Proposition 4.4(i).
So, it follows from the commutative diagram
{ $$
\begin{CD}
\widehat{H}_{n-1}(\bd\,\overline U;G)@>{{\widehat{i}^{n-1}_{\bd\,\overline U,\overline U}}}>>\widehat{H}_{n-1}(\overline U;G)\\
@VV{T^{n-1}_{\bd\,\overline U}}V@VV{T^{n-1}_{\overline U}}V\\
H_{n-1}(\bd\,\overline U;G)@>{{i^{n-1}_{\bd\,\overline U,\overline U}}}>> H_{n-1}(\overline U;G)@.
\end{CD}
$$}\\
that $T^{n-1}_{bd\,\overline U}(L)$ contains a non-trivial element $\gamma$ with
$i^{n-1}_{bd\,\overline U,\overline U}(\gamma)=0$. The last equality implies $i^{n-1}_{bd\,\overline U,X}(\gamma)=0$. Thus, $n\in\mathcal H_{X,G}$.

$(3)$ We have $\dim_GX=n$ (recall that $X$ is dimensionally full-valued). Moreover, by \cite{ha}, $h\dim_GX=\dim_GX=n$. Hence, as in the proof of item (2), we can find a point $x\in X$ such that $\widehat{H}_n(\overline U,bd\,\overline U;G)\neq 0$ for all sufficiently small neighborhoods $U$ of $x$. Because $X$ is $lc^{n}$, $X$ is also locally homologically connected up to dimension $n$ with respect to \v{C}ech homology with arbitrary coefficients. Hence, there exists a neighborhood $V$ of $x$ such that the homomorphism $i^n_{V,X}: H_n(V;G)\to H_n(X;G)$ is trivial.

We claim that $\widehat{H}_n(\overline U;G)=0$ for each neighborhood $U$ of $x$ with $\overline U\subset V$. Indeed, suppose there is a neighborhood $U$ of $x$ with $\overline U\subset V$ and $\widehat{H}_n(\overline U;G)\neq 0$. Then Proposition 4.4(i) implies that $\widehat{H}_n(\overline U;G)$ is isomorphic to $H_n(\overline U;G)$. So, $H_n(\overline U;G)$ is a non-trivial group such that the inclusion homomorphism $i^n_{\overline U,X}=i^n_{V,X}\circ i^n_{\overline U,V}$ is trivial. That yields $d_GX\geq n+1$, which contradicts Proposition 2.1.

Therefore, the homomorphism $\partial:\widehat{H}_{n}(\overline U,bd\,\overline U;G)\to\widehat{H}_{n-1}(bd\,\overline U;G)$ is injective and
the image $L=\partial(\widehat{H}_{n}(\overline U,bd\,\overline U;G))$ is a non-trivial subgroup of $\widehat{H}_{n-1}(bd\,\overline U;G)$ with
$\widehat{i}^{n-1}_{bd\,\overline U,X}(L)=0$ (see the proof of item (3)). Because the groups $\widehat{H}_{n-1}(bd\,\overline U;G)$ and $\widehat{H}_{n-1}(X;G)$ are isomorphic to $H_{n-1}(bd\,\overline U;G)$ and
$H_{n-1}(X;G)$, respectively, we have  $d_GX=n$.

To prove that $d_\mathbb ZX=n$, consider the field $\mathbb Q$ of all rational numbers. According to the previous paragraph, there is a point $x\in X$ with
$H_{n-1}(bd\,\overline U;\mathbb Q)\neq 0$ for all sufficiently small neighborhoods $U$ of $x$. So, by Proposition 4.5, $H_{n-1}(bd\,\overline U;\mathbb Z)\neq 0$ for any such $U$. Finally, Proposition 4.9 (together with the equality $\dim X=n$) implies $d_\mathbb ZX=n$.
\end{proof}

\begin{cor}
Let $X$ be a metric $lc^{n}$-compactum, where $\dim X=n$.
Then the following conditions are equivalent: $X$ is dimensionally full-valued; there is a point $x\in X$
having a local base $\mathcal B_x$ such that $bd\,\overline U$ is dimensionally full-valued for each $U\in\mathcal B_x$;
$d_GX=n$ for any torsion free field $G$; $d_\mathbb ZX=n$.
%\item[(2)] If $\dim_\mathbb QX=n$, then $d_\mathbb QX=n$.
%\end{itemize}
\end{cor}
Let us also mention the following corollary, which improves a result of  Kodama \cite[Theorem 8]{ko1}.
\begin{cor}
Every two-dimensional $lc^{2}$-compactum is dimensionally full-valued.
\end{cor}
\begin{proof}
The complete proof is presented in \cite[Theorem 3.2]{v}, here we provide a sketch only. By \cite{a}, $d_{\mathbb S^1}X=h\dim_{\mathbb S^1}X=2$.
Then, as in the proof of Theorem 2.9(3), there is $x\in X$ such that $\widehat{H}_2(\overline U,\bd\,\overline U;\mathbb S^1)$ and $\widehat{H}_{1}(\bd\,\overline U;\mathbb S^1)$ are both non-trivial for sufficiently small neighborhoods $U$ of $x$, while the homomorphisms
%$\widehat{i}^{1}_{\overline U,X}:\widehat{H}_{1}(\overline U;\mathbb S^1)\to\widehat{H}_{1}(X;\mathbb S^1)$ and
$\widehat{i}^{1}_{\bd\,\overline U,X}:\widehat{H}_{1}(\bd\,\overline U;\mathbb S^1)\to\widehat{H}_{1}(X;\mathbb S^1)$
are trivial. Moreover, $\widehat{i}^{2}_{\overline U,X}:\widehat{H}_{2}(\overline U;\mathbb S^1)\to\widehat{H}_{2}(X;\mathbb S^1)$ is also trivial for small $U$ because $X$ is $lc^{2}$ and the exact homology groups with coefficients $\mathbb S^1$ are isomorphic to the corresponding \v{C}ech groups (see Proposition 4.4(i)).
 So, by the Universal Coefficient Theorem for \v{C}ech cohomology, $H^2(\overline U,\bd\,\overline U;\mathbb Z)\neq 0$, $H^1(\bd\,\overline U;\mathbb Z)\neq 0$ and the triviality of the homomorphism $\widehat{i}^{1}_{\bd\,\overline U,X}$
%$\widehat{i}^{2}_{\overline U,X}$ and $\widehat{i}^{1}_{\overline U,X}$
yields the triviality of the inclusion homomorphism
%$j^2_{X\overline U}:H^2(X;\mathbb Z)\to H^2(\overline U;\mathbb Z)$ and
$j^1_{X\bd\,\overline U}:H^1(X;\mathbb Z)\to H^1(\bd\,\overline U;\mathbb Z)$. %Since $\dim X=2$, there exists a homomorphism from $H^{2}(X,\bd\,\overline U;\mathbb Z)$ onto $H^{2}(\overline U,\bd\,\overline U;\mathbb Z)$. So, $H^{2}(X,\bd\,\overline U;\mathbb Z)\neq 0$ for all small $U$.
It follows from the exact sequence
{ $$
\begin{CD}
\cdots\to H^{1}(X;\mathbb Z)@>{{j^1_{X\bd\,\overline U}}}>>H^1(\bd\,\overline U;\mathbb Z)@>{{\partial_X}}>>H^2(X,\bd\,\overline U;\mathbb Z)\to\cdots
\end{CD}
$$}\\
that $\partial_X$ is an injective homomorphism. Hence, $H^2(X,\bd\,\overline U;\mathbb Z)$ contains elements of infinite order (recall that
the simplicial one-dimensional cohomology groups with integer coefficients are free, so any non-trivial one dimensional \v{C}ech cohomology group $H^1(.:\mathbb Z)$ is torsion free, see for example \cite[Theorem 12.5]{dr}).
This implies that $H^2(X,\bd\,\overline U;\mathbb Q)\neq 0$.
 So, $\dim_{\mathbb Q}X=2$, and by \cite[p.364]{ha}, $h\dim_{\mathbb Q}X=2$. On the other hand, $d_{\mathbb Q}X=h\dim_{\mathbb Q}X=2$, see \cite[Corollary 2.3]{v}.
Finally, Theorem 2.9(1) yields $X$ is dimensionally full-valued.
\end{proof}
%\begin{proof}
%The first item follows directly from Theorem 2.9. The second item requires more considerations. Since $\dim_\mathbb QX=n$, by
%\cite[Theorem 2(3)]{ku}, there exists a closed set $K\subset X$ and a point $x\in K$ such that $H^{n-1}(bd_K\,\overline W;\mathbb Q)\neq 0$ for all sufficiently small neighborhoods $W$ of $x$ in $K$. By the universal coefficients formula for \v{C}ech cohomology, $H^{n-1}(bd_K\,\overline W;\mathbb Q)$ is isomorphic to $H^{n-1}(bd_K\,\overline W;\mathbb Z)\otimes\mathbb Q$, so $H^{n-1}(bd_K\,\overline W;\mathbb Z)$ contains elements of infinite order. Consequently, $\mathrm{Hom}(H^{n-1}(F_W;\mathbb Z),\mathbb Q)\neq 0$ (see \cite[\S 43, Example 6]{fu}), where $F_W=bd_K\,\overline W$.
%Therefore, it follows from the exact sequence
%$$0\to\Ext(H^{n}(F_W;\mathbb Z),\mathbb Q)\to\widehat{H}_{n-1}(F_W;\mathbb Q)\to \Hom(H^{n-1}(F_W;\mathbb Z),\mathbb Q)\to 0$$
%that $\widehat{H}_{n-1}(F_W;\mathbb Q)\neq 0$.
%$\mathrm{Hom}(H^{n-1}(F_V;\mathbb Z),\mathbb Q)\neq 0$ using only the equality $\dim_\mathbb QX=\dim X=n$. Then the exact coefficient formulas involving $F_V$ (see the proof of Theorem 2.9(3)) implies $\widehat{H}_{n-1}(F;\mathbb Q)\neq 0$. Hence, according to Proposition 2.8, $n\in\mathcal H_{X,\mathbb Q}$ (recall that
%$\widehat{H}_{n-1}(F;\mathbb Q)$ is isomorphic to $H_{n-1}(F;\mathbb Q)$).
 % Finally, by Proposition 2.1, $d_\mathbb QX=n$.
%\end{proof}

%\begin{proof}
%Suppose $n\not\in\mathcal{H}_{X,G}$.
%\end{proof}

\section{Homogeneous metric $ANR$-compacta}
The following statement was established in \cite{vv0} (see Theorem 1.1 and Corollary 1.2):
\begin{pro}
Let $X$ be a finite dimensional homogeneous $ANR$-space with $\dim X\geq 2$. Then $X$ has the following properties for any abelian group $G$ and $n\geq 2$ with $n\in\mathcal{H}_{X,G}$ and $n+1\not\in\mathcal{H}_{X,G}$:
\begin{itemize}
\item[(1)] Every $x\in X$ has a basis of open sets $U_k$ such that $H_{n-1}(\overline U_k;G)=0$ and  $H_{n-1}(\bd\,\overline U_k;G)\neq 0$;
\item[(2)] If a closed subset $K\subset X$ is an $(n-1)$-homology membrane spanned on $B$ for some closed set $B\subset X$ and $\gamma\in H_{n-1}(B;G)$, then $(K\setminus B)\cap\overline{X\setminus K}=\varnothing$.
\end{itemize}
\end{pro}

Here is our first proposition concerning homogeneous $ANR$.
\begin{pro}
Let  $X$ be as in Proposition $3.1$ and $Z\subset X$ a closed set. If $n\geq 2$ is an integer with $n\in\mathcal{H}_{Z,G}$ and $n+1\not\in\mathcal{H}_{X,G}$, then
$Z$ has a non-empty interior in $X$. Moreover $d_GZ=d_GX$ iff $\Int(Z)\neq\varnothing$.
\end{pro}

\begin{proof}
Since $n\in\mathcal H_{Z,G}$, there exists a closed set $F\subset Z$ and non-trivial $\gamma\in H_{n-1}(F,G)$ homologous to zero in $Z$. Consequently, by Corollary 2.7, we can find a closed set $K\subset Z$ containing $F$ such that $K$ is a homological membrane for $\gamma$ spanned on $F$. Because $\gamma$, being homologous to zero in $Z$, is homologous to zero in $X$, $n\in\mathcal H_{X,G}$. Recall that $n+1\not\in\mathcal H_{X,G}$, so we can apply Proposition 3.1(2) to conclude that
$(K\setminus F)\cap\overline{X\setminus K}=\varnothing$. This implies that $K\setminus F$ is open in $X$, i.e., $\Int(Z)\neq\varnothing$.

Suppose now that $d_GZ=d_GX=n$. Then $n\in\mathcal H_{Z,G}$ and $n+1\not\in\mathcal{H}_{X,G}$. Hence, $\Int(Z)\neq\varnothing$. Conversely, if $\Int(Z)\neq\varnothing$ and
$d_GX=n$, choose $x\in\Int(Z)$. Then, by Proposition 3.1(1), there exists $U\in\mathcal B_x$ such that $\overline U\subset\Int(Z)$ and $H_{n-1}(\bd\,\overline U;G)$
contains a non-trivial element homologous to zero in $\overline U$. Hence, $d_GZ\geq n$. Finally, the inequality $d_GZ\leq d_GX$ implies $d_GZ=d_GX$. $\Box$
\end{proof}

\begin{thm}\label{main1}
Let $X$ be a homogeneous $ANR$-space.
\begin{itemize}
\item[(1)] If $G$ is a field and $\dim_GX=n$, then $n\in\mathcal H_{X,G}$ and $n+1\not\in\mathcal H_{X,G}$;
\item[(2)] If $\dim X=n$, then $\mathcal H_{X,\mathbb Z}\cap\{n-1,n\}\neq\varnothing$.
\end{itemize}
\end{thm}

\begin{proof}
$(1)$ Since $G$ is a field, by \cite{ha}, $h\dim_GX=\dim_GX=n$. Then there is a point $x\in X$ having sufficiently small open neighborhoods $U\subset X$ with $\widehat{H}_n(\overline U,bd\,\overline U;G)\neq 0$ (see the arguments from Theorem 2.9(2)).
Consider the exact sequence
$$\to\widehat{H}_{n}(\overline U;G)\to\widehat{H}_{n}(\overline U,bd\,\overline U;G)\to\widehat{H}_{n-1}(bd\,\overline U;G)\to\widehat{H}_{n-1}(\overline U;G).$$
%$$0\to\mathrm{Ext_G}(H^{n+1}(\overline U;G),G)\to\widehat{H}_{n}(\overline U;G)\to\mathrm{Hom_G}(H^{n}(\overline U;G),G)\to 0.$$
We claim that $\widehat{H}_{n}(\overline U;G)=0$ for all sufficiently small $U$. Indeed, since
 the groups $H_{n}(\overline U;G)$ and $\widehat{H}_{n}(\overline U;G)$ are isomorphic by Proposition 4.4(i) (recall that $G$ is a field), and $X$ is locally contractible, we can choose $\overline U$ so small that the inclusion homomorphism $i^n_{\overline U,X}:H_{n}(\overline U;G)\to H_{n}(X;G)$ is trivial. So,
  $\widehat{H}_{n}(\overline U;G)\neq0$ would imply $d_GX\geq n+1$, which contradicts Proposition 2.1. Hence, $\widehat{H}_{n}(\overline U;G)=0$. Then, proceeding as
 in the proof of Theorem 2.9(2), we obtain $n\in\mathcal H_{X,G}$.
Finally, by Proposition 2.1, $d_GX\leq\dim_GX=n$. Therefore, $n+1\not\in\mathcal H_{X,G}$.

$(2)$ By \cite[Theorem 1.1]{vv1}, every point $x\in X$ has a basis of neighborhoods $U$ with $\dim bd\,\overline U=n-1$ and $H^{n-1}(bd\,\overline U;\mathbb Z)\neq 0$. For any such $U$ consider the exact sequences, where the coefficient group $\mathbb Z$ is suppressed
$$0\to\mathrm{Ext}(H^{n}(bd\,\overline U),\mathbb Z)\to\widehat{H}_{n-1}(bd\,\overline U)\to\mathrm{Hom}(H^{n-1}(bd\,\overline U),\mathbb Z)\to 0$$
and
$$0\to\mathrm{Ext}(H^{n-1}(bd\,\overline U),\mathbb Z)\to\widehat{H}_{n-2}(bd\,\overline U)\to\mathrm{Hom}(H^{n-2}(bd\,\overline U),\mathbb Z)\to 0.$$
The equality $\widehat{H}_{n-1}(bd\,\overline U;\mathbb Z)=\widehat{H}_{n-2}(bd\,\overline U;\mathbb Z)=0$ implies the triviality of the groups $\mathrm{Ext}(H^{n-1}(bd\,\overline U;\mathbb Z),\mathbb Z)$
and $\mathrm{Hom}(H^{n-1}(bd\,\overline U;\mathbb Z),\mathbb Z)$. Thus, $H^{n-1}(bd\,\overline U;\mathbb Z)=0$ (see \S52, Exercise 9 from \cite{fu}), a contradiction. Hence, at least one of the groups $\widehat{H}_{n-1}(bd\,\overline U;\mathbb Z)$ and $\widehat{H}_{n-2}(bd\,\overline U;\mathbb Z)$ is not trivial. According to \cite[Theorem 1.1]{vv1}, $H^{n-1}(bd\,\overline U;\mathbb Z)$ is finitely generated. Hence, by \cite[Theorem 4]{sk} (see also Proposition 4.4(iv) below), $\widehat{H}_{n-2}(bd\,\overline U;\mathbb Z)$ is isomorphic to
 $H_{n-2}(bd\,\overline U;\mathbb Z)$. Moreover, the equality $\dim bd\,\overline U=n-1$ yields that $\widehat{H}_{n-1}(bd\,\overline U;\mathbb Z)$ is isomorphic to $H_{n-1}(bd\,\overline U;\mathbb Z)$. Because $X$ is an $ANR$, the sets $bd\,\overline U$ are contractible in $X$ for sufficiently small $U$. Therefore, both homomorphisms $i^{n-1}_{bd\,\overline U,X}$ and
 $i^{n-2}_{bd\,\overline U,X}$ are trivial, which implies that $\mathcal H_{X,\mathbb Z}$ contains at least one of the integers $n-1$ and $n$.
\end{proof}

\section{Appendix}
There are two types of Universal Coefficients Theorems for the exact homology, see Theorem 2 and Theorem 3 from \cite{sk}.
\begin{pro}
Let $(X,A)$ be a pair of compact metric spaces and $G$ be a module over a principal ideal domain $K$. Then we have the exact sequence:
$$0\to\Ext_K(H^{n+1}(X,A;K),G)\to\widehat{H}_{n}(X,A;G)\to \Hom_K(H^n(X,A;K),G)\to 0.$$
If, in addition, $G$ is finitely generated, the sequence
$$0\to\widehat{H}_{n}(X,A;K)\otimes
_KG\to\widehat{H}_{n}(X,A;G)\to\widehat{H}_{n-1}(X,A;G)*_KG\to 0$$ is also exact, where
$\widehat{H}_{n-1}(X,A;K)*_KG$ denotes the torsion product $\Tor_K(\widehat{H}_{n-1}(X,A;K),G)$.
\end{pro}

We say that a space $X$ is {\em homologically locally connected up to dimension $n$} (br., $lc^n$) if for every $x\in X$ and a neighborhood $U$ of $x$ in $X$ there exists a neighborhood $V\subset U$ of $x$ such that the homomorphism $\widetilde{i}_{V,U}^m:\widetilde{H}_m(V;\mathbb Z)\to\widetilde{H}_m(U;\mathbb Z)$ is trivial for all $m\leq n$. Here, $\widetilde{H}_*(.;.)$ denotes
 the singular homology groups.
 According to \cite[Theorem 1]{mar}, the following is true:
\begin{pro}
Let $(X,A)$ be a pair of paracompact spaces. If both $X$ and $A$ are $lc^n$, then there exists a natural transformation $M_{X,A}$ between the singular and
\v{C}ech homologies of $(X,A)$ such that $M^k_{X,A}:\widetilde{H}_k(X,A;G)\to H_k(X,A;G)$ is an isomorphism for each $k\leq n$ and each group $G$.
\end{pro}

There is also an analogue of Proposition 4.3 concerning the homology $\widehat{H}_{*}(X,A;G)$, see \cite[Proposition 9]{sk}:
\begin{pro}
Let $(X,A)$ be a pair of compact metric spaces with both $X$ and $A$ being $lc^n$. Then there is a natural transformation $S_{X,A}$ between the singular and the exact homologies of $(X,A)$ such that $S^k_{X,A}:\widetilde{H}_k(X,A;G)\to\widehat{H}_k(X,A;G)$ is an isomorphism for each $k\leq n-1$ and it is surjective for $k=n$.
\end{pro}
As we already noted, if $G$ is a module over a ring with unity and $(X,A)$ is a compact metric pair, then there exists a natural transformation $$T_{X,A}:\widehat{H}_{*}(X,A;G)\to H_{*}(X,A;G)$$ between the exact and \v{C}ech homologies such that the homomorphism $T^k_{X,A}:\widehat{H}_{k}(X,A;G)\to H_{k}(X,A;G)$ is surjective for each $k$. According to \cite[Theorem 4]{sk}, this homomorphism is an isomorphism under certain conditions. We list below  some of these conditions.
\begin{pro}
The homomorphism $T^k_{X,A}$ is an isomorphism in each of the following situations:
\begin{itemize}
\item[(i)] The group $G$ admits a compact topology or $G$ is a vector space over a field;
\item[(ii)] Both $\widehat{H}_{k}(X,A;G)$ and $G$ are countable modules;
\item[(iii)] $\dim X=k$;
\item[(iv)] both $H^{k+1}(X,A;\mathbb Z)$ and $G$ are finitely generated with finite number of relations;
\end{itemize}
\end{pro}

We don't know if the second Universal Coefficient Theorem from Proposition 4.1 also holds for \v{C}ech homologies when $K=\mathbb Z$ and $G$ is a torsion free field. Probably $\widehat{H}_{n}(X,A;G)$ is not isomorphic to $\widehat{H}_{n}(X,A;K)\otimes\mathbb G$, but we have the following conclusion:
\begin{pro}
Let $(X,A)$ be a compact pair and $n$ be a non-negative integer. Then $H_{n}(X,A;G)\neq 0$ implies $H_{n}(X,A;\mathbb Z)\neq 0$ for every torsion free abelian group $G$.
\end{pro}
\begin{proof}
We choose a family $\{\omega_\alpha\}$ of finite open covers of $X$ such that $H_{n}(X,A;G)$ is the limit of the inverse system $\{H_{n}(N_\alpha^X,N_\alpha^A;G),\pi^*_{{\beta},\alpha}\}$, where $N_\alpha^X$ and $N_\alpha^A$ are the nerves of $\omega_\alpha$ and the restriction of $\omega_\alpha$ over $A$, and $\pi_{{\beta},\alpha}:(N_{\beta}^X,N_{\beta}^A)\to (N_\alpha^X,N_\alpha^A)$ are the corresponding simplicial maps with $\beta$ being a refinement of $\alpha$. By Proposition 4.2, all $H_{n}(N_\alpha^X,N_\alpha^A;G)$ (resp., $H_{n}(N_\alpha^X,N_\alpha^A;\mathbb Z)$) are isomorphic to the singular homology groups $\widetilde{H}_{n}(N_\alpha^X,N_\alpha^A;G)$ (resp., $\widetilde{H}_{n}(N_\alpha^X,N_\alpha^A;\mathbb Z)$). So, the second Universal Coefficient Theorem from Proposition 4.1, which holds for the singular homology with arbitrary coefficients, implies that
$H_{n}(N_\alpha^X,N_\alpha^A;G)$ are isomorphic to the tensor products $H_{n}(N_\alpha^X,N_\alpha^A;\mathbb Z)\otimes G$ (recall that, since $G$ is torsion free, the torsion product of $G$ with any abelian group is trivial). Because $H_{n}(X,A;\mathbb Z)$ is the
limit of the inverse sequence $\{H_{n}(N_\alpha^X,N_\alpha^A;\mathbb Z),\pi^*_{\beta,\alpha}\}$, the triviality of $H_{n}(X,A;\mathbb Z)$ would yields
that $H_{n}(N_\alpha^X,N_\alpha^A;\mathbb Z)=0$ for all $\alpha$. Consequently, all groups $H_{n}(N_\alpha^X,N_\alpha^A;G)$ would be also trivial, which contradicts the condition $H_{n}(X,A;G)\neq 0$.
\end{proof}

\begin{pro}
Let $(X,A)$ be a compact pair, $G$ be an abelian group  and $n\geq 1$ an integer. If there is a non-trivial element $\gamma\in H_{n-1}(A;G)$ such that
$i^{n-1}_{A,X}(\gamma)=0$, then $H_n(X,A;G)\neq 0$.
\end{pro}

\begin{proof}
Following the notations from the proof of Proposition 4.5, take a family $\{\omega_\alpha\}$ of finite open covers of $X$ such that the homology groups $H_{n}(X,A;G)$,
$H_{n-1}(X;G)$ and $H_{n-1}(A;G)$ are limit, respectively, of the inverse systems $\{H_{n}(N_\alpha^X,N_\alpha^A;G),\pi^*_{{\beta},\alpha}\}$,
$\{H_{n-1}(N_\alpha^X;G),\pi^*_{{\beta},\alpha}\}$ and $\{H_{n-1}(N_\alpha^A;G),\pi^*_{{\beta},\alpha}\}$. If  for each $\alpha$ we denote by
$\pi_\alpha:A\to N_\alpha^A$ the natural map, then there exists $\alpha_0$ such that $\gamma_{\alpha_0}=\pi_{\alpha_0}^*(\gamma)$ is a non-trivial
element of $H_{n-1}(N^A_{\alpha_0};G)$. It follows from the commutative diagram
{ $$
\begin{CD}
H_{n-1}(A;G)@>{{i^{n-1}_{A,X}}}>>H_{n-1}(X;G)\\
@VV{\pi^*_{\alpha_0}}V@VV{\pi^*_{\alpha_0}}V\\
H_{n-1}(N^A_{\alpha_0};G)@>{{i^{n-1}_{N^A_{\alpha_0},N^X_{\alpha_0}}}}>> H_{n-1}(N^X_{\alpha_0};G)@.
\end{CD}
$$}\\
that $i^{n-1}_{N^A_{\alpha_0},N^X_{\alpha_0}}(\gamma_{\alpha_0})=0$. Because $H_{n}(N_{\alpha_0}^X,N_{\alpha_0}^A;G)$,
$H_{n-1}(N^A_{\alpha_0};G)$ and $H_{n-1}(N^X_{\alpha_0};G)$ are naturally isomorphic to the corresponding singular homology groups, we have the exact sequence
$$
\xymatrix{
\to H_{n}(N_{\alpha_0}^X,N_{\alpha_0}^A;G)\to H_{n-1}(N^A_{\alpha_0};G)\to H_{n-1}(N^X_{\alpha_0};G)\to \ldots\\
}
$$
Therefore, $H_{n}(N_{\alpha_0}^X,N_{\alpha_0}^A;G)\neq 0$, which implies $H_n(X,A;G)\neq 0$.
\end{proof}

Let us also mention the following sufficient condition for a given metric compactum to be dimensionally full-valued, see \cite[Corollary 1]{ko}.
We provide a short proof of this result, different from the original one.
\begin{pro}
Let $X$ be a space with $\dim X=n$. If there exists a closed subset $A\subset X$ such that $H_n(X,A;\mathbb Z)\neq 0$, then $X$ is
dimensionally full-valued.
\end{pro}

\begin{proof}
Since $\dim X=n$, $H_{n}(X,A;\mathbb Z)$ and $\widehat{H}_{n}(X,A;\mathbb Z)$ are isomorphic and $H^{n+1}(X,A;\mathbb Z)=0$. So, the
first exact sequence from Proposition 4.1 implies $\Hom(H^n(X,A;\mathbb Z),\mathbb Z)\neq 0$. This means that $H^n(X,A;\mathbb Z)$ contains $\mathbb Z$ as a
direct multiple. Thus,  the tensor product $H^n(X,A;\mathbb Z)\otimes G$ is not trivial for all abelian groups $G$. Then the exact sequence
$$0\to H^n(X,A;\mathbb Z)\otimes G\to H^n(X,A;G)\to H^n(X,A;\mathbb Z)*G\to 0$$ yields $H^n(X,A;G)\neq 0$. Therefore, $\dim_GX=\dim X$ for all abelian
groups $G$. Finally, by \cite[Theorem 11]{ku}, $X$ is dimensionally full-valued.
\end{proof}

\begin{pro}
If $X$ is $k-lc$ and $(k-1)-lc$, where $k\geq 1$, then $X$ is $k-lc$ with respect to every abelian group $G$.
\end{pro}

\begin{proof}
Let $U$ be a neighborhood of a given point $x$. Since $X$ is $k-lc$ and $(k-1)-lc$, there exist two neighborhoods $W\subset V$ of $x$ such that
$V\subset U$ and the inclusion homomorphisms $\widetilde{i}^p_{W,V}:\widetilde{H}_p(W;\mathbb Z)\to\widetilde{H}_p(V;\mathbb Z)$,
$\widetilde{i}^p_{V,U}:\widetilde{H}_p(V;\mathbb Z)\to\widetilde{H}_p(U;\mathbb Z)$ are trivial for every $p=k-1,k$. Consider the following diagram,
whose rows are exact sequences

$$
\xymatrix{
&\widetilde{H}_k(W;\mathbb Z)\otimes G\ar[rr]^{}\ar[d]_{\widetilde{i}^k_{W,V}\otimes id}&&\widetilde{H}_k(W;G)\ar[rr]^{}\ar[d]_{\widetilde{i}^k_{W,V}}&&
\widetilde{H}_{k-1}(W;\mathbb Z)*G\ar[d]_{\widetilde{i}^{k-1}_{W,V}*id}\\
&\widetilde{H}_k(V;\mathbb Z)\otimes G\ar[rr]^{}\ar[d]_{\widetilde{i}^k_{V,U}\otimes id}&&\widetilde{H}_k(V;G)\ar[rr]^{}\ar[d]_{\widetilde{i}^k_{V,U}}&&
\widetilde{H}_{k-1}(V;\mathbb Z)*G\ar[d]_{\widetilde{i}^{k-1}_{V,U}*id}\\
&\widetilde{H}_k(U;\mathbb Z)\otimes G\ar[rr]^{}&&\widetilde{H}_k(U;G)\ar[rr]^{}&&
\widetilde{H}_{k-1}(U;\mathbb Z)*G\\
}
$$
Because the homomorphisms $\widetilde{i}^{k}_{V,U}\otimes id$ and $\widetilde{i}^{k-1}_{W,V}*id$ are trivial, so is the composition
$\widetilde{i}^{k}_{V,U}\circ\widetilde{i}^{k}_{W,V}=\widetilde{i}^{k}_{W,U}$. Hence, $X$ is $k-lc$ with respect to $G$.
\end{proof}

\begin{pro}
Let $X$ be a metric $lc^{n-1}$-space with the following property: there exists a point $x\in X$ such that any sufficiently small neighborhood $V$ of $x$ contains a closed set $F_V\subset X$ with $H_{n-1}(F_V;G)\neq 0$, where $G$ is a given abelian group. Then $n\in\mathcal H_{X,G}$.
\end{pro}

\begin{proof}
Let $G$ be a fixed group. Then $X$ is $k-lc$ with respect to $G$ for all $k\leq n-1$ (for $k\geq 1$ this follows from Proposition 4.8, and for $k=0$ it is obvious). On the other hand, by \cite[Theorem 1]{mar}, the singular homology groups $\widetilde{H}_{n-1}(W;G)$ are isomorphic to the groups $H_{n-1}(W;G)$ for
every open set $W\subset X$. Thus, the inclusion homomorphisms $i^{n-1}_{U,X}:H_{n-1}(U;G)\to H_{n-1}(X;G)$ are trivial for all small
neighborhoods $U$ of $x$.
Then any such $U$ contains $F_U$ and the homomorphism $i^{n-1}_{F_U,X}:H_{n-1}(F_U;G)\to H_{n-1}(X;G)$, being the composition $i^{n-1}_{U,X}\circ i^{n-1}_{F_U,U}$, is trivial.
Therefore, $n\in\mathcal H_{X,G}$.
\end{proof}

\textbf{Acknowledgments.} The author would like to express his gratitude to K. Kawamura for his helpful comments. The author also thanks the referee for his/her careful reading and suggesting many improvements of the paper.

\end{document}